\documentclass[12pt]{article}
\usepackage{a4wide}
\usepackage{amsthm}
\usepackage{amsfonts}
\usepackage{amssymb}
\usepackage{amsmath}
\usepackage{cite}
\usepackage{epsfig}
\usepackage{tabularx}

\newtheorem{theorem}{Theorem}

\newtheorem{lemma}[theorem]{Lemma}
\newtheorem{proposition}[theorem]{Proposition}

\newcommand\lbound{{7}}

\DeclareTextCompositeCommand{\v}{OT1}{l}{l\nobreak\hspace{-.1em}'}
\DeclareTextCompositeCommand{\v}{OT1}{t}{t\nobreak\hspace{-.1em}'\nobreak\hspace{-.15em}}

\begin{document}
\title{Planar graph with twin-width seven\thanks{Both authors have been supported by the MUNI Award in Science and Humanities (MUNI/I/1677/2018) of the Grant Agency of Masaryk University.}}

\author{Daniel Kr{\'a}\v{l}\thanks{Faculty of Informatics, Masaryk University, Botanick\'a 68A, 602 00 Brno, Czech Republic. E-mail: {\tt \{dkral,lamaison\}@fi.muni.cz}.}\and
\newcounter{lth}
\setcounter{lth}{2}
        Ander Lamaison$^\fnsymbol{lth}$}
\date{} 
\maketitle

\begin{abstract}
We construct a planar graph with twin-width equal to seven.
\end{abstract}

\section{Introduction}
\label{sec:intro}

Twin-width is a graph parameter,
which has recently been introduced by Bonnet, Kim, Thomass\'e and Watrigant~\cite{BonKTW20,BonKTW22} and
further developed in particular in~\cite{BonGKTW21b,BonGKTW21c,BonGOSTT22,BonGOT22,BonKRT22,BonGTT22,BonCKKLT22},
also see~\cite{Tho22}.
The notion has rapidly become an important notion in theoretical computer science,
which is witnessed by a large number of recent papers exploring its algorithmic aspects~\cite{BerBD22,BonGKTW21c,BonGOSTT22,BonKRTW21,PilSZ22},
combinatorial properties~\cite{AhnHKO22,BalH21,BonGKTW21c,BonKRT22,DreGJOR22,PilS22}, and
connections to logic and model theory~\cite{BonGOSTT22,BonKTW20,BonKTW22,BonNOST21,GajPT22}.

Classes of graphs with bounded twin-width (we refer to Section~\ref{sec:notation} for the definition of the parameter)
include both classes of sparse graphs and classes of dense graphs,
in particular, classes of graphs with bounded treewidth, with bounded rank-width (or equivalently with bounded clique-width), and
excluding a fixed graph as a minor have bounded twin-width.
As the first order model checking is fixed parameter tractable for classes of graphs with bounded twin-width~\cite{BonKTW20,BonKTW22},
the notion led to a unified view of various earlier results of fixed parameter tractability of first order model checking of graph properties, and
more generally first order model checking properties of other combinatorial structures such as
matrices, permutations and posets~\cite{BalH21,BonGOSTT22,BonNOST21}.

One of the most important graph classes is that of planar graphs.
Since the class of planar graphs is a proper minor-closed class of graphs,
planar graphs have bounded twin-width by general results obtained in the seminal paper of Bonnet et al.~\cite{BonKTW20,BonKTW22},
also see~\cite{BonGKTW21b,DvoHJLW21}.
The first explicit bounds have been obtained by Bonnet, Kwon and Wood~\cite{BonKW22} and Jacob and Pilipczuk~\cite{JacP22},
who showed that the twin-width of planar graphs is at most 583 and 183, respectively.
The bound was further improved to 37 by Bekos, Lozzo, Hlin\v en\'y and Kaufmann~\cite{BekLHK22}, 
to $9$ by Hlin\v en\'y~\cite{Hli22}, and
eventually to $8$ by Hlin\v en\'y and Jedelsk\'y~\cite{HliJ22}; also see~\cite{Hli23}.
On the other hand, any planar graph with minimum degree five, no adjacent vertices of degree $5$ and no separating triangles
has twin-width at least $5$.
In this paper,
we present a first non-trivial lower bound on the twin-width of planar graphs
by constructing a planar graph with twin-width equal to $7$.

\section{Notation}
\label{sec:notation}

We now fix notation used throughout the paper.
Unless we use the term \emph{multigraph},
graphs considered in this paper are simple i.e., without parallel edges and loops;
on the other hand, multigraphs may contain parallel edges and loops.
A \emph{$k$-vertex} is a vertex of degree $k$,
a \emph{$(\le k)$-vertex} is a vertex of degree at most $k$, and
a \emph{$(\ge k)$-vertex} is a vertex of degree at least $k$.
Similarly, a \emph{$k$-cycle} is a cycle of length $k$, and
a \emph{$(\le k)$-cycle} is a cycle of length at most $k$.

We next formally define the notion of twin-width.
A \emph{trigraph} is a graph with some of its edges being red;
the \emph{red degree} of a vertex $v$ is the number of red edges incident with $v$.
If $G$ is a trigraph and $v$ and $v'$ form a pair of its (not necessarily adjacent) vertices,
then the trigraph obtained from $G$ by \emph{contracting} the vertices $v$ and $v'$
is the trigraph obtained from $G$ by removing the vertices $v$ and $v'$ and introducing a new vertex $w$ such that
$w$ is adjacent to each vertex $u$ that is adjacent to at least one of the vertices $v$ and $v'$ in $G$ and
the edge $wu$ is red if $u$ is not adjacent to both $v$ and $v'$ or at least one of the edges $vu$ and $v'u$ is red,
i.e., the edge $wu$ is not red only if $G$ contains both edges $vu$ and $v'u$ and neither of the two edges is red.
A \emph{$k$-contraction} is a contraction of vertices such that the resulting graph has maximum red degree at most $k$.
The \emph{twin-width} of a graph $G$ is the smallest integer $k$ such that
there exists a sequence of $k$-contractions that reduces the graph $G$,
i.e., the trigraph with the same vertices and edges as $G$ and no red edges,
is reduced by the sequence to a single vertex.

We conclude the section with fixing notation related to plane graphs and
recalling a classical result on light edges,
which we later need in our arguments.
A \emph{plane multigraph} is an embedding of a multigraph in the plane.
The \emph{length} of a face of a connected plane multigraph is the length of the facial walk (counting twice each edge visited twice by the facial walk), and
a \emph{$k$-face} is a face of length $k$.
Finally, a cycle $C$ of a plane (multi)graph $G$ is \emph{separating}
if both the interior and the exterior of the cycle $C$ contains a vertex of $G$.
We will also need the following result on the existence of edges joining vertices of small degrees due to Borodin~\cite{Bor89},
also see~\cite{CraW17}.

\begin{proposition}
\label{prop:normal}
Every plane multigraph with minimum degree at least three and minimum face length at least three
contains an edge $e$ such that the sum of the degrees of the end vertices of $e$ is at most $11$, or 
a $4$-cycle passing through two $3$-vertices and a common $(\le 10)$-vertex.
\end{proposition}

Proposition~\ref{prop:normal} yields the following.

\begin{proposition}
\label{prop:edge11}
Every plane multigraph with maximum degree at most seven and with minimum face length at least three
contains an edge $e$ such that the sum of the degrees of the end vertices of $e$ is at most $11$.
\end{proposition}

\section{Construction and analysis}
\label{sec:main}

In this section, we present a construction of a planar graph with twin-width seven and
analyze the constructed graph to show that its twin-width is indeed seven.

\subsection{Construction}

\begin{figure}
\begin{center}
\epsfbox{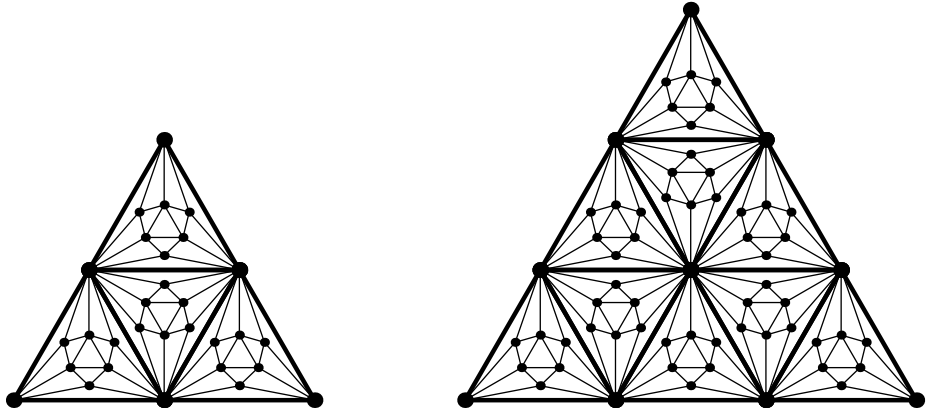}
\end{center}
\caption{The graphs placed in each triangular face of the icosahedron for $k=1$ and $k=2$
         in the construction of the graph $G_k$.
	 The edges of the skeleton are drawn bold.}
\label{fig:Gk-triang}
\end{figure}

Recall that the vertices and edges of the icosahedron form a $5$-regular planar graph,
which has $12$ vertices, $30$ edges and $20$ faces.
Let $G_k$ be the graph obtained from the icosahedron graph as follows:
\begin{itemize}
\item subdivide each edge exactly $k$ times,
\item paste the appropriate piece of the triangular lattice in each of the $20$ faces,
\item insert a triangle into each of the $20(k+1)^2$ faces and
      join each of its three vertices to a different vertex of the original face, and
\item insert a vertex into each of the $60(k+1)^2$ faces of size four and
      join it to all four vertices on the boundary of the face.
\end{itemize}
The graphs created in each triangular face of the icosahedron graph for $k=1$ and $k=2$
are depicted in Figure~\ref{fig:Gk-triang}.
Observe that the graph $G_k$ has
\begin{itemize}
\item $60(k+1)^2$ vertices of degree four,
\item $60(k+1)^2$ vertices of degree five,
\item $12$ vertices of degree twenty, and
\item $30k+20\cdot\frac{k(k-1)}{2}=10k(k+2)$ vertices of degree twenty four.
\end{itemize}

The vertices of degree at least twenty of $G_k$ are referred to as \emph{skeleton} vertices, and
the other vertices as \emph{non-skeleton} vertices.
The \emph{skeleton} of $G_k$ is the subgraph induced by the skeleton vertices.
More generally, if a graph $H$ is obtained by a sequence of contractions from $G_k$,
we refer to a vertex $v$ of $H$ as a \emph{skeleton} vertex
if at least one of the vertices contracted to $v$ is a skeleton vertex, and
as a \emph{non-skeleton} vertex otherwise.
A \emph{skeleton neighbor} of a vertex $v$ is a neighbor of $v$ that is a skeleton vertex.

The \emph{skeleton} of a graph $H$, which is obtained from $G_k$ by contracting some of its vertices,
is the graph formed by the skeleton vertices of $H$,
where two vertices $x$ and $y$ are adjacent iff there exist skeleton vertices $x'$ and $y'$ of $G_k$ such that
$x'$ is $x$ or has been contracted to $x$, $y'$ is $y$ or has been contracted to $y$, and $x'y'$ is an edge of $G_k$.
The \emph{skeleton degree} of a skeleton vertex is the number of its neighbors in the skeleton of $H$.
In particular, the skeleton degree of every vertex of $G_k$ is either five or six.
In general, the skeleton degree of a vertex can be smaller than the number of its skeleton neighbors.

A contraction of two vertices of the graph $H$ is \emph{semiplanar}
if either at least one of the vertices is a non-skeleton vertex, or
both vertices are skeleton vertices and
they are joined by an edge of the skeleton of $H$ not contained in a separating $3$-cycle of the skeleton.
Observe that a semiplanar contraction either preserves the skeleton or contracts an edge of it, and
the skeleton of a graph obtained from $G_k$ by a sequence of semiplanar contractions is a triangulation.
In particular, each face of the skeleton of any of graph obtained from a graph $G_k$ by a sequence of semiplanar contractions
is a $3$-face and corresponds to a face of the skeleton of $G_k$ (but not necessarily vice versa).

\subsection{Upper bound}

We start with an auxiliary lemma.
Note that the lemma implies that
the twin-width of any plane triangulation with maximum degree at most seven is at most seven.

\begin{lemma}
\label{lm:contract7}
Let $H$ be a plane triangulation with maximum degree at most seven.
There exists a sequence of edge-contractions that reduces $H$ to a single vertex such that
all graphs obtained during the process have maximum degree at most seven.
\end{lemma}

\begin{proof}
We prove by induction a stronger statement when $H$ is allowed to be a plane multigraph without $2$-faces.
In particular,
when contracting an edge during the induction step,
we remove one of two parallel edges in each of the two arising $2$-faces,
however, we keep the remaining parallel edges, i.e., parallel edges that do not bound a face.

The base of the induction is the case that $H$ is a triangle.
For the proof of the induction step, consider a plane triangulation $H$ with maximum degree at most seven.
If $H$ has a vertex of degree at most two, we contract an edge incident with $v$;
note that if $v$ has degree two, we remove two of the three resulting parallel edges.
If the minimum degree of $H$ is at least three,
we contract any edge $uv$ such that the sum of the degrees of $u$ and $v$ is at most $11$;
such an edge exists by Proposition~\ref{prop:edge11}.
Note that the degree of the vertex created by the contraction of the edge $uv$ is at most $11-4=7$ (here,
we use that all faces are $3$-faces), and the degree of no other vertex increases.
\end{proof}

We next show that the twin-width of $G_k$ is at most seven.

\begin{lemma}
\label{lm:upper}
For every $k\ge 0$, the twin-width of $G_k$ is at most seven.
\end{lemma}

\begin{proof}
We present a sequence of $7$-contractions that reduces $G_k$ to a single vertex.
First, consider one face of the skeleton $G_k$ after another, and
for each face $f$ of the skeleton,
fix a non-skeleton $5$-vertex $v$ contained in $f$ and contract the remaining non-skeleton vertices to $v$ one after another.
Let $H_1$ be the resulting graph;
note that $H_1$ consists of the skeleton of $G_k$ with a single vertex inserted to each face,
which is joined by red edges to all three skeleton vertices bounding the face.
Throughout the whole process of contractions transforming $G_k$ to $H_1$,
each skeleton vertex has at most one red neighbor in each face of the skeleton and so its red degree is at most six.
During the process of contractions within a face $f$ of the skeleton,
the only red edges are incident with the chosen vertex $v$ in the face $f$;
after the first contraction within the face $f$ is performed,
the vertex $v$ has at most seven neighbors:
three skeleton vertices bounding the face $f$ and at most four remaining non-skeleton vertices contained in $f$.
Hence, the red degree of $v$ does not exceed seven (and
the red degree of any other non-skeleton vertex contained in $f$ is at most one as its red neighbor can be $v$ only).
We conclude that $H_1$ has been obtained by a sequence of $7$-contractions from $G_k$.

Let $v_1,\ldots,v_n$ be the skeleton vertices of $H_1$ (listed in an arbitrary order); note that $n=10k^2+20k+12$.
Consecutively for $i=1,\ldots,n$, 
contract the non-skeleton vertices contained in the faces bounded by $v_i$ and two of the vertices among $v_{i+1},\ldots,n$ 
to a single vertex and then contract the resulting vertex to the vertex $v_i$.
Note that all other non-skeleton vertices adjacent to $v_i$ have already been contracted to skeleton neighbors of $v_i$
before the $i$-th step.
Let $H_2$ be the resulting graph, which coincides with the skeleton of $G_k$.
Observe that throughout the whole process of contractions that transforms $H_1$ to $H_2$,
each of the skeleton vertices has at most six red neighbors and
each non-skeleton vertex adjacent to $v_i$ in the $i$-th step
has degree at most seven (the vertex $v_i$ and its at most six skeleton neighbors), and
so its red degree is at most seven.

Note that, since $H_2$ is the skeleton of $G_k$, the maximum degree of $H_2$ is six.
Lemma~\ref{lm:contract7} implies that there exists a sequence of edge-contractions that
eventually reduces $H_2$ to a single vertex while all graphs obtained during the process have maximum degree at most seven.
Hence, we conclude that the twin-width of the graph $G_k$ is at most seven.
\end{proof}

\subsection{Lower bound}

We now proceed with showing that the twin-width of $G_k$ is equal to seven for $k\ge\lbound$.
To do so, we will present a series of lemmas,
which concern graphs that can be obtained by a sequence of $6$-contractions from the graph $G_k$.
We start with a lemma that states that all initial contractions of skeleton vertices involve adjacent vertices.
Note that the conditions of the statement of the lemma are satisfied by the graph $G_k$ in particular;
they are also satisfied after a limited number of contractions of vertices of $G_k$.

\begin{lemma}
\label{lm:cont-non}
Let $H$ be a graph obtained by a sequence of semiplanar $6$-contractions from $G_k$ for $k\ge\lbound$.
If the skeleton of $H$ has minimum degree five,
every pair of $5$-vertices of the skeleton has at most one common skeleton neighbor and
the skeleton of $H$ has no separating $(\le 4)$-cycles,
then there is no $6$-contraction of skeleton vertices of $H$ that are not adjacent in the skeleton.
\end{lemma}

\begin{proof}
Let $v_1$ and $v_2$ be two non-adjacent vertices of the skeleton of $H$.
If at least one of the vertices $v_1$ and $v_2$ has skeleton degree six,
then the vertex obtained by their contraction has red degree at least $5+6-4=7$ (as they have at most two common skeleton neighbors  and the skeleton of $H$ has no separating $(\le 4)$-cycle).
If both vertices $v_1$ and $v_2$ have skeleton degree five,
then the vertex obtained by their contraction has red degree at least $5+5-2=8$ (as they are not adjacent and
have at most one common skeleton neighbor).
We conclude that there is no $6$-contraction of vertices of the skeleton of $H$ that are not adjacent in the skeleton.
\end{proof}

The next lemma excludes the existence of $6$-contractions of adjacent $6$-vertices of the skeleton,
again under the assumptions satisfied by $G_k$ and by graphs obtained by a limited number of contractions from $G_k$.

\begin{lemma}
\label{lm:cont66}
Let $H$ be a graph obtained by a sequence of semiplanar $6$-contractions from $G_k$ for $k\ge\lbound$.
Assume that the skeleton of $H$ has minimum degree five and has no two adjacent $5$-vertices.
Let $v_1$ and $v_2$ be two adjacent $6$-vertices of the skeleton of $H$ such that
$v_1$ and $v_2$ have exactly two common neighbors in the skeleton and
at most one of their (not necessarily common) skeleton neighbors has skeleton degree five.
The contraction of $v_1$ and $v_2$ results in a vertex of red degree at least seven.
\end{lemma}

\begin{proof}
Suppose that
the $6$-vertices $v_1$ and $v_2$ of the skeleton of $H$ that
have the properties given in the statement of the lemma
can be $6$-contracted, and
let $w$ be the vertex resulting from their contraction in $H$.
By Lemma~\ref{lm:cont-non}, the vertices $v_1$ and $v_2$ are adjacent in the skeleton of $H$, and
let $u_T$ and $u_B$ be the two common skeleton neighbors of $v_1$ and $v_2$.
Also see Figure~\ref{fig:cont66} for illustration of the notation.
Further,
let $\alpha_T$, $\beta_T$, $\gamma_T$ be the three non-skeleton vertices forming a triangle inside the face $v_1v_2u_T$ of the skeleton
such that $\alpha_T$ is a neighbor of $v_1$, $\beta_T$ is a neighbor of $v_2$ and $\gamma_T$ is a neighbor of $u_T$.
Similarly, 
let $\alpha_B$, $\beta_B$, $\gamma_B$ be the three non-skeleton vertices forming a triangle inside the face $v_1v_2u_B$
such that $\alpha_B$ is a neighbor of $v_1$, $\beta_B$ is a neighbor of $v_2$ and $\gamma_B$ is a neighbor of $u_B$.

\begin{figure}
\begin{center}
\epsfbox{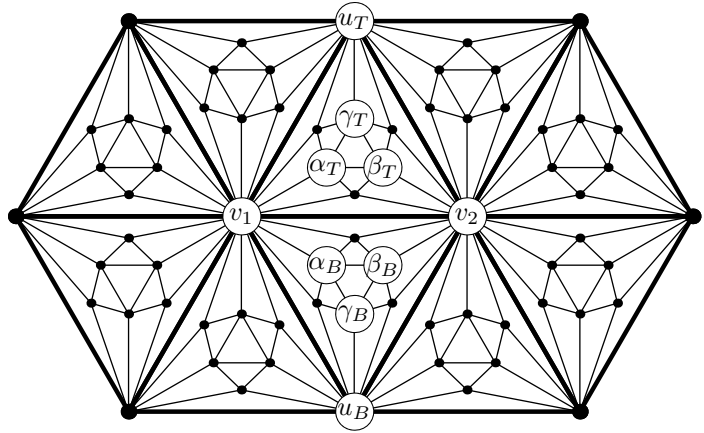}
\end{center}
\caption{Notation used in the proof of Lemma~\ref{lm:cont66}.}
\label{fig:cont66}
\end{figure}

We first establish that, in the graph $H$, all vertices $\alpha_T$, $\beta_T$, $\alpha_B$, $\beta_B$
have been contracted to a skeleton neighbor of $v_1$ or $v_2$ that
is different from any of the vertices $v_1$, $v_2$, $u_B$ and $u_T$.
By symmetry, it is enough to consider the vertex $\alpha_T$.
First observe that the vertex $w$ has eight skeleton neighbors and
the six neighbors different from $u_B$ and $u_T$ are red neighbors,
i.e., it cannot have any additional red neighbor.
If the vertex $\alpha_T$ was contracted to a skeleton vertex that is not a neighbor of $v_1$ or $v_2$, or
if the vertex $\alpha_T$ was contracted to $u_B$ or $u_T$,
then $w$ would have seven red skeleton neighbors.
Similarly,
if the vertex $\alpha_T$ was not contracted to a skeleton vertex,
then the neighbor of $w$ that is the vertex $\alpha_T$ or the vertex to which $\alpha_T$ has been contracted
would be a red neighbor of $w$ in addition to the six skeleton red neighbors of $w$.
Finally, if the the vertex $\alpha_T$ was contracted to $v_1$ or $v_2$,
then all skeleton neighbors of $w$ would be red and the red degree of $w$ would be at least eight.
Hence, the vertex $\alpha_T$ must have been contracted to a skeleton neighbor of $v_1$ or $v_2$ that
is different from any of the vertices $v_1$, $v_2$, $u_B$ and $u_T$.

Let $a_T$ and $b_T$ be the vertices of $H$ to which the vertices $\alpha_T$ and $\beta_T$ have been contracted, respectively.
If $a_T$ and $b_T$ are different non-adjacent vertices of the skeleton of $H$,
then at least one of the vertices $a_T$ and $b_T$ has skeleton degree six, and
this vertex has red degree at least seven when the vertices $v_1$ and $v_2$ have been contracted (its red neighbors
are all its six skeleton neighbors and the other of the vertices $a_T$ and $b_T$).
If follows that $a_T$ and $b_T$ are adjacent in the skeleton and
so they are skeleton neighbors of the same among the vertices $v_1$ and $v_2$;
by symmetry, we may assume that both $a_T$ and $b_T$ are skeleton neighbors of $v_1$.
If the skeleton degree of $b_T$ in $H$ is six,
then $b_T$ has at least seven red neighbors in $H$:
the six skeleton neighbors and the vertex $v_2$.
Hence, the skeleton degree of $b_T$ is five.

Similarly,
let $a_B$ and $b_B$ be the vertices of $H$ to which the vertices $\alpha_B$ and $\beta_B$ have been contracted, respectively.
Since there is only one skeleton neighbor of $v_1$ and $v_2$ with skeleton degree five,
we conclude that $b_B=b_T$ (in addition,
the vertex $a_B$ is either $b_B$ or a common skeleton neighbor of $b_B$ and $v_1$,
which is different from $u_B$ and $u_T$).

If the vertex $\gamma_T$ has not been contracted to a skeleton vertex of $H$,
then $b_T$ has red degree at least seven:
its red neighbors are the five skeleton neighbors, the vertex $v_2$ and
either the vertex $\gamma_T$ or the non-skeleton vertex to which it has been contracted to.
Hence, the vertex $\gamma_T$ has been contracted to a skeleton vertex of $H$ and let $c_T$ be this vertex.
Note that $c_T$ is neither $v_1$ nor $v_2$ (otherwise,
the vertex $u_B$ would be a red neighbor of $w$ and so the red degree of $w$ would be at least seven).
Similarly, $c_T$ is neither $u_T$ nor $u_B$ (otherwise, the red degree of $w$ would be at least seven).
Suppose that $c_T$ is not the vertex $b_T$.
This yields that $c_T$ is a skeleton neighbor of $b_T$ in $H$ (otherwise, $b_T$ would have red degree at least seven in $H$ as
the five skeleton neighbors of $b_T$, the vertex $v_2$ and $c_T$ would be its red neighbors),
which implies that the skeleton degree of $c_T$ is six.
Since the vertices $c_T$ and $u_T$ are different (as we have argued earlier),
$c_T$ is a skeleton neighbor of $u_T$ (otherwise, $c_T$ would have red degree at least seven in $H$ as
its six skeleton neighbors and the vertex $u_T$ would be its red neighbors).
We conclude that $c_T$ is either the vertex $b_T=b_B$ or a common skeleton neighbor of $b_T$ and $u_T$.

Along the same lines, we argue that the vertex $\gamma_B$ has been contracted to a skeleton vertex of $H$ and
this skeleton vertex $c_B$ is either the vertex $b_T=b_B$ or a common skeleton neighbor of $b_T$ and $u_B$.
It follows that $b_T=b_B$ is the skeleton neighbor of $v_1$ that is not a skeleton neighbor of $u_B$ or $u_T$,
$c_T$ is the unique common skeleton neighbor of $b_T$ and $u_T$, and
$c_B$ is the unique common skeleton neighbor of $b_T$ and $u_B$.
Also see Figure~\ref{fig:cont66b}.

\begin{figure}
\begin{center}
\epsfbox{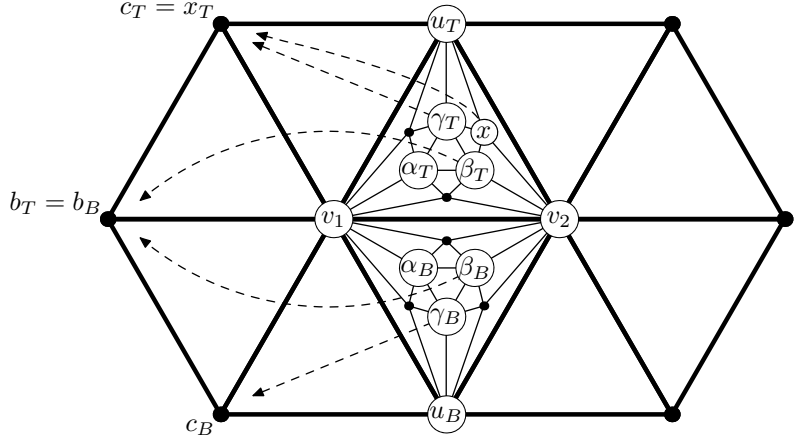}
\end{center}
\caption{Contractions of the non-skeleton vertices to the skeleton as established in the proof of Lemma~\ref{lm:cont66}.}
\label{fig:cont66b}
\end{figure}

Let $x$ be the common neighbor of $v_2$, $u_T$, $\beta_T$ and $\gamma_T$ in $G_k$.
If the vertex $x$ was not contracted to a skeleton vertex of $H$,
then the vertex $b_T=b_B$ would have red degree at least seven in $H$:
its red neighbors would be its five skeleton neighbors, the vertex $v_2$ and 
the vertex $x$ or the non-skeleton vertex to which the vertex $x$ has been contracted to.
Let $x_T$ be the skeleton vertex of $H$ to which the vertex $x$ has been contracted to.
The vertex $x_T$ is neither of the vertices $v_1$, $v_2$, $u_T$ or $u_B$ (otherwise,
the vertex $w$ would have red degree at least seven).
The vertex $x_T$ is not the vertex $b_T=b_B$ (otherwise, $b_T$ would have red degree at least seven as
its red neighbors would be its five skeleton neighbors, the vertex $v_2$ and the vertex $u_T$).
However, the vertex $x_T$ is a skeleton neighbor of $b_T=b_B$ (otherwise, $b_T$ would have red degree at least seven as
its red neighbors would be its five skeleton neighbors, the vertex $v_2$ and the vertex $x_T$) and
it is also a skeleton neighbor of $v_1$ or $v_2$ (otherwise, the vertex $w$ would have red degree at least seven).
Since the vertex $x_T$ cannot be the vertex $c_B$ (otherwise, $c_B$ would have red degree at least seven as
its red neighbors would be its five skeleton neighbors different from $u_B$, $v_2$ and $u_T$),
the vertex $x_T$ must be the vertex $c_T$.

We next distinguish whether the vertex $x$ was contracted to the vertex $c_T$ before or after
the vertex $\beta_T$ was contracted to $b_T$.
If the vertex $x$ was contracted to $c_T$ before the vertex $\beta_T$ was contracted to $b_T$,
the red degree of $c_T$ would be at least seven at that point:
the red neighbors of $c_T$ would be its six skeleton neighbors possibly except for $u_T$,
the vertex $v_2$ and the vertex $\beta_T$ or the vertex to which $\beta_T$ had been contracted.
If the vertex $\beta_T$ was contracted to $b_T$ before the vertex $x$ was contracted to $c_T$,
the red degree of $b_T$ would be at least seven at that point:
the red neighbors of $b_T$ would be its five skeleton neighbors,
the vertex $v_2$ and the vertex $x$ or the vertex to which $x$ had been contracted.
Hence, neither of the two cases can apply, and
we conclude that a sequence of semiplanar $6$-contractions of $G_k$
cannot result in a graph $H$ with the properties given in the statement of the lemma such that
the vertices $v_1$ and $v_2$ can be $6$-contracted.
\end{proof}

The next lemma describes when an adjacent $5$-vertex and $6$-vertex of the skeleton can be $6$-contracted,
again under the assumptions satisfied by $G_k$ and by graphs obtained by a limited number of contractions from $G_k$.
The cases described in the statement of the lemma are illustrated in Figure~\ref{fig:cont56-statement}.

\begin{figure}
\begin{center}
\epsfbox{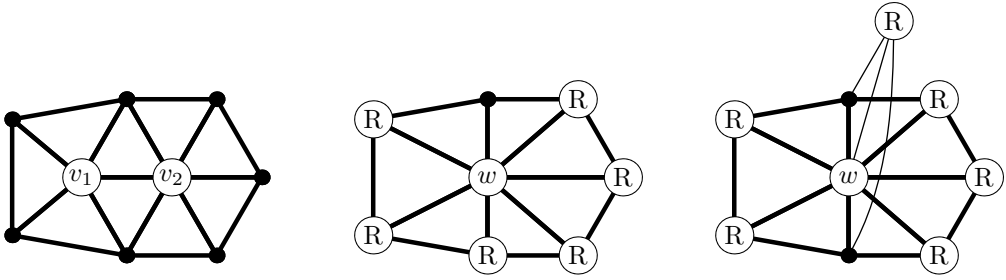}
\end{center}
\caption{The skeleton of the graph $H$ and the two possible outcomes described in the statement of Lemma~\ref{lm:cont56};
         the red neighbors of the newly created vertex $w$ are marked by R.}
\label{fig:cont56-statement}
\end{figure}

\begin{lemma}
\label{lm:cont56}
Let $H$ be a graph obtained by a sequence of semiplanar $6$-contractions from $G_k$ for $k\ge\lbound$.
Assume that the skeleton of $H$ has minimum degree five and has no separating $3$-cycles.
Let $v_1$ and $v_2$ be an adjacent $5$-vertex and $6$-vertex of the skeleton such that
each vertex at distance at most two from $v_1$ and $v_2$ in the skeleton of $H$ is a $(\ge 6)$-vertex.
If it is possible to $6$-contract $v_1$ and $v_2$,
then the resulting vertex $w$ has red degree six and exactly one of the following applies:
\begin{itemize}
\item the red neighbors of $w$ are six skeleton neighbors of $v_1$ and $v_2$,
      one of the two common skeleton neighbors of $v_1$ and $v_2$ is not a red neighbor of $w$, and
      the other common skeleton neighbor of $v_1$ and $v_2$, which we call $u_0$, satisfies that
      each of its skeleton neighbors is its red neighbor,
\item the red neighbors of $w$ are the five skeleton neighbors of $v_1$ and $v_2$ that are not common skeleton neighbors of $v_1$ and $v_2$, and
      a non-skeleton vertex $u_0$ that is adjacent to $w$ and the two common neighbors of $v_1$ and $v_2$ in the skeleton of $H$, and
      all neighbors of $u_0$ in $H$ are red.
\end{itemize}
\end{lemma}

\begin{proof}
Let $u_T$ and $u_B$ denote the two common skeleton neighbors of $v_1$ and $v_2$;
$\alpha_T$, $\beta_T$, $\gamma_T$ the three non-skeleton vertices forming a triangle inside the face $v_1v_2u_T$ of the skeleton
such that $\alpha_T$ is a neighbor of $v_1$, $\beta_T$ is a neighbor of $v_2$ and $\gamma_T$ is a neighbor of $u_T$, and
$\alpha_B$, $\beta_B$, $\gamma_B$ the three non-skeleton vertices forming a triangle inside the face $v_1v_2u_B$
such that $\alpha_B$ is a neighbor of $v_1$, $\beta_B$ is a neighbor of $v_2$ and $\gamma_B$ is a neighbor of $u_B$.
Also see Figure~\ref{fig:cont56} for illustration of the notation.
Suppose that the contraction of $v_1$ and $v_2$ in $H$ is a $6$-contraction, and
let $w$ be the vertex obtained by contracting $v_1$ and $v_2$ (as in the statement of the lemma).

\begin{figure}
\begin{center}
\epsfbox{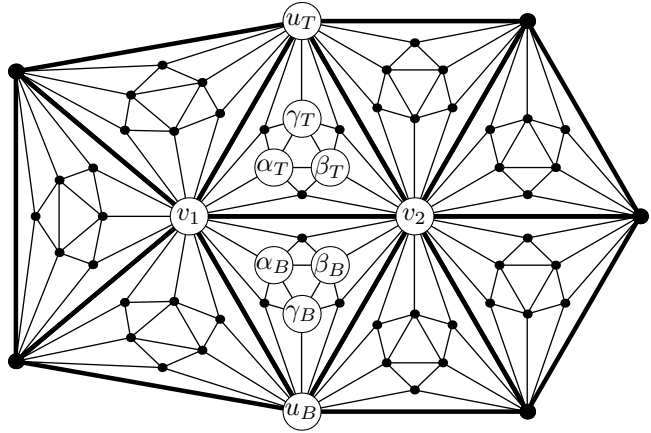}
\end{center}
\caption{Notation used in the proof of Lemma~\ref{lm:cont56}.}
\label{fig:cont56}
\end{figure}

Observe that the five skeleton neighbors of $w$ different from $u_T$ and $u_B$ are red neighbors of $w$.
We first consider the case that one of the vertices $u_T$ and $u_B$ is a red neighbor of $w$.
By symmetry, we may assume that $u_B$ is a red neighbor of $w$ and so $u_T$ is not.
The vertex $u_B$ will be the vertex $u_0$ from the first case described in the statement of the lemma.
Since $w$ has six red skeleton neighbors,
both vertices $\alpha_T$ and $\beta_T$ have been contracted to skeleton neighbors of $w$ different from $u_T$.
Let $a_T$ and $b_T$ be the skeleton vertices to which $\alpha_T$ and $\beta_T$ have been contracted, respectively.
Note that $a_T$ is a neighbor of $v_1$ (otherwise, $a_T$ would be neighbor of $v_2$ different from $u_B$ and $u_T$ and
so $a_T$ would have have red degree seven in $H$,
in particular, the red neighbors of $a_T$ would be $v_1$ and the six skeleton neighbors of $a_T$) and,
similarly, $b_T$ is a neighbor of $v_2$.
If $a_T$ and $b_T$ are different vertices that are not adjacent in the skeleton of $H$,
then one of them, say $a_T$, is not the vertex $u_B$, and
the red degree of $a_T$ after the contraction of $v_1$ and $v_2$ is seven (all its six skeleton neighbors and
the vertex $b_T$ are its red neighbors).
Hence, either $a_T$ and $b_T$ are the same vertex and this vertex is $u_B$, or
the vertices $a_T$ and $b_T$ are adjacent in the skeleton of $H$ and so one of them is the vertex $u_B$ (as
the vertex $a_T$ is a neighbor of $v_1$ and $b_T$ is a neighbor of $v_2$).
In either of the cases,
at least one of the vertices $\alpha_T$ and $\beta_T$ has been contracted to $u_B$ and
all five skeleton neighbors of the vertex $u_B$
are red neighbors of $u_B$ after the contraction of $v_1$ and $v_2$ (as required by the first case in the statement of the lemma).

We next consider the case that neither of the vertices $u_T$ and $u_B$ is a red neighbor of $w$.
Let $a_T$, $b_T$ and $c_T$ be the vertices that
$\alpha_T$, $\beta_T$ and $\gamma_T$ have been contracted to in $H$ (possibly,
$a_T$ is $\alpha_T$, $b_T$ is $\beta_T$, and $c_T$ is $\gamma_T$).
Similarly,
let $a_B$, $b_B$ and $c_B$ be the vertices that
$\alpha_B$, $\beta_B$ and $\gamma_B$ have been contracted to in $H$.
Observe that none of the vertices $a_T$, $b_T$, $c_T$, $a_B$, $b_B$ and $c_B$
is one of the vertices $v_1$, $v_2$, $u_T$ and $u_B$ (otherwise, $w$ would have $u_T$ or $u_B$ as a red skeleton neighbor).

Suppose that $a_T$ is a skeleton neighbor of $w$.
Hence, $a_T$ must be a skeleton neighbor of $v_1$ in $H$ (otherwise, $a_T$ would have red degree seven in $H$).
If $b_T$ is a skeleton neighbor of $a_T$, then $b_T$ has red degree at least seven in $H$, and
if $b_T$ is not a skeleton neighbor of $a_T$, then $a_T$ has red degree at least seven after contracting $v_1$ and $v_2$.
Therefore, $a_T$ is not a skeleton neighbor of $w$.
Along the same lines,
we conclude that neither of the vertices $b_T$, $a_B$ and $b_B$ is a skeleton neighbor of $w$.

If any two of the vertices $a_T$, $b_T$, $a_B$ and $b_B$ are distinct,
then $w$ has red degree seven:
there are five red skeleton neighbors of $w$ and
the two distinct vertices among $a_T$, $b_T$, $a_B$ and $b_B$ are also red neighbors of $w$.
It follows that $a_T=b_T=a_B=b_B$ and we denote this vertex $u_0$ in the rest of the proof.
Note that we have established that $u_0$ is not a neighbor of $v_1$ or $v_2$ in the skeleton of $H$.
If $u_0$ is a skeleton vertex of $H$, then $u_0$ has red degree at least seven:
its red neighbors are its at least five skeleton neighbors and
the vertices $v_1$ and $v_2$ are also its red neighbors.
Hence, $u_0$ is not a skeleton vertex.
Since the vertices $\alpha_T$, $\beta_T$, $\alpha_B$ and $\beta_B$ do not have any common neighbor,
\emph{all neighbors of the vertex $u_0$ in $H$ are red}.

We next show that the vertex $u_0$ is adjacent to the vertex $u_T$ in $H$.
Suppose that $u_0$ is not adjacent to the vertex $u_T$.
Let $z_A$ be the vertex of $H$ to which the common neighbor of $v_1$, $u_T$, $\alpha_T$ and $\gamma_T$ has been contracted in $H$, and
let $z_B$ be the vertex of $H$ to which the common neighbor of $v_2$, $u_T$, $\beta_T$ and $\gamma_T$ has been contracted in $H$.

If $z_A$ is not a skeleton vertex of $H$,
then $z_A$ is the vertex $u_0$ (otherwise, $w$ would have red degree at least seven), and
so $u_0$ is adjacent to $u_T$.
Hence, $z_A$ is a skeleton vertex of $H$.
Note that $z_A$ cannot be one of the vertices $v_1$ and $v_2$ as $u_B$ is not a red neighbor of $w$.
If $z_A$ is the vertex $u_T$, then $u_0$ is adjacent to $u_T$ as desired.
So, we can assume that $z_A$ is none of the vertices $v_1$, $v_2$ and $u_T$.
If $z_A$ is not a neighbor of $u_T$ in the skeleton of $H$,
then the red degree of $z_A$ in $H$ is at least seven:
the red neighbors of $z_A$ are the at least five neighbors of $z_A$ in the skeleton of $H$,
the vertex $u_T$ and the vertex $u_0$.
If $z_A$ is not a neighbor of $v_1$ in the skeleton of $H$,
then the red degree of $z_A$ in $H$ is at least seven:
the red neighbors of $z_A$ are the at least five neighbors of $z_A$ in the skeleton of $H$,
the vertex $v_1$ and the vertex $u_0$.
It follows that $z_A$ is the common neighbor of $v_1$ and $u_T$ in the skeleton of $H$ that is different from $v_2$.

A symmetric argument yields that
if $u_0$ is not adjacent to $u_T$,
then $z_B$ is the common neighbor of $v_2$ and $u_T$ in the skeleton of $H$ that is different from $v_1$.
See Figure~\ref{fig:cont56z}.

\begin{figure}
\begin{center}
\epsfbox{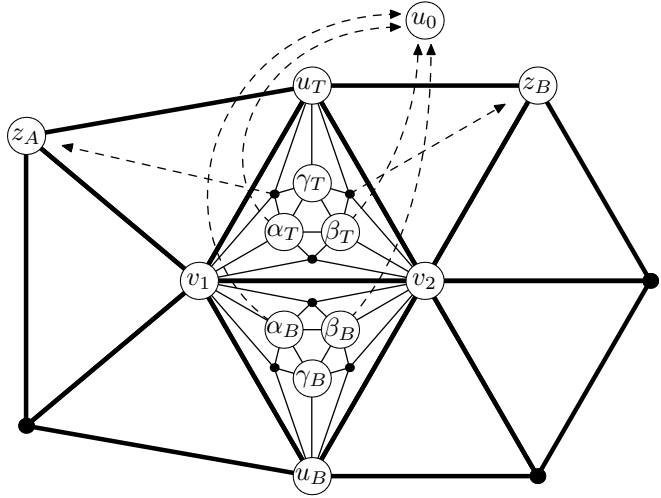}
\end{center}
\caption{Illustration of the notation used in the analysis presented in the proof of Lemma~\ref{lm:cont56}.}
\label{fig:cont56z}
\end{figure}

We now consider which of the vertices of $H$ is the vertex $c_T$,
i.e., the vertex to which $\gamma_T$ has been contracted in $H$.
If $c_T$ is not a skeleton vertex, then $z_A$ has red degree at least seven after the contraction of $v_1$ and $v_2$:
the red neighbors of $z_A$ are its skeleton neighbors different from $u_T$ and the vertices $u_0$ and $c_T$ (note that
$c_T$ cannot be the vertex $u_0$ as otherwise $u_0$ is adjacent to $u_T$ as desired).
Hence, the vertex $c_T$ is one of the skeleton vertices of $H$.
Note that $c_T$ is neither $v_1$ and $v_2$ as $u_B$ is not a red neighbor of $w$.
If $c_T$ is a vertex of the skeleton of $H$ that is not a common neighbor of $z_A$ and $z_B$ in the skeleton,
i.e, $c_T$ is a skeleton vertex different from $u_T$,
then the red degree of $c_T$ is at least seven:
its red neighbors are at least five neighbors of $c_T$ in the skeleton of $H$ different from $u_T$ (note that
if $c_T$ has skeleton degree five, then $u_T$ is not its skeleton neighbor),
the vertex $u_0$, and the vertex $z_A$ or $z_B$ that is not a neighbor of $c_T$ in the skeleton.
Hence, the vertex $c_T$ is the vertex $u_T$ and so the vertex $u_0$ is adjacent to $u_T$ in $H$.

An analogous argument yields that the vertex $u_0$ is adjacent to $u_B$ in $H$,
which completes the analysis of the second case described in the statement of the lemma.
\end{proof}

The next lemma describes
when a $5$-vertex that is a common neighbor of $6$-contracted adjacent $5$-vertex and $6$-vertex of the skeleton can be $6$-contracted,
again under the assumptions satisfied by $G_k$ and by graphs obtained by a limited number of contractions from $G_k$.
We refer to Figure~\ref{fig:cont556-statement} for the illustration of the notation used in the statement of the next lemma.
Note that the statement of the lemma does not require $v'_1$ to be a $5$-vertex of the skeleton of $G_k$,
i.e., $v'_1$ could become a $5$-vertex of the skeleton by some $6$-contractions applied to $G_k$.

\begin{figure}
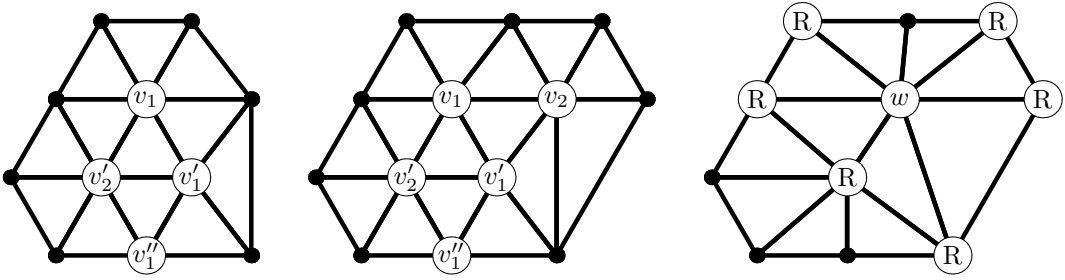

\begin{center}
\epsfbox{twin7-7.mps}
\hskip 5mm
\epsfbox{twin7-8.mps}
\hskip 5mm
\epsfbox{twin7-9.mps}
\end{center}
\caption{The notation used in the statement of Lemma~\ref{lm:cont556};
         the only possible configuration before the contraction of the vertices $v'_1$ and $v'_2$ to $w'$ and the vertices $v_1$ and $v_2$ to $w$, and
	 the configuration resulting from the contraction (the red neighbors of $w$ are depicted by R).}
\label{fig:cont556-statement}
\end{figure}

\begin{lemma}
\label{lm:cont556}
Let $H$ be a graph obtained by a sequence of semiplanar $6$-contractions from $G_k$ for $k\ge\lbound$.
Assume that the skeleton of $H$ has minimum degree five and has no separating $3$-cycles.
Let $v_1$ and $v_2$ be adjacent $5$-vertex and $6$-vertex of the skeleton of $H$ such that
each vertex at distance at most two from $v_1$ and $v_2$ in the skeleton of $H$ is a $6$-vertex
except for a vertex $w'$, which is a neighbor of $v_1$ and
which is a $7$-vertex of the skeleton obtained by contracting a $5$-vertex $v'_1$ and a $6$-vertex $v'_2$ of the skeleton, and
possibly except for a vertex $v''_1$ of the skeleton,
which was either a $6$-vertex or a $7$-vertex of the skeleton adjacent to both $v'_1$ and $v'_2$ before the contraction of $v'_1$ and $v'_2$.
In addition,
assume that no $6$-vertex that is a neighbor of $v_1$ and $v_2$ in the skeleton of $H$
has been obtained by a contraction of two or more skeleton vertices, and
the conclusion of Lemma~\ref{lm:cont56} is satisfied just after the contraction of the vertices $v'_1$ and $v'_2$
with respect to the contraction of these two vertices.

If it is possible to $6$-contract $v_1$ and $v_2$,
then
\begin{itemize}
\item $w'$ is a common neighbor of $v_1$ and $v_2$ in the skeleton of $H$,
\item the vertex $w'$ has red degree six and its red neighbors are all its skeleton neighbors except for $v_1$,
\item all skeleton neighbors of $v''_1$ are red,
\item the vertex $w$ resulting from the contraction of $v_1$ and $v_2$ has red degree six,
its red neighbors are its neighbors in the skeleton except the common neighbor of $v_1$ and $v_2$ different from $w'$, and
\item just before the contraction of $v'_1$ and $v'_2$, 
the vertex $v'_1$ was a common neighbor of $v_1$ and $v_2$ in the skeleton and
all its neighbors in the skeleton possibly except for $v_1$ were red.
\end{itemize}
In particular,
after the contraction of $v_1$ and $v_2$,
the first case in the conclusion of Lemma~\ref{lm:cont56} with respect to $v_1$ and $v_2$,
i.e., the red degree of $w$ is six and its only skeleton neighbor that is not red
is the common neighbor of $v_1$ and $v_2$ different from $w'$, and
all skeleton neighbors of $w'$ are red.
\end{lemma}

\begin{proof}
Since the conclusion of Lemma~\ref{lm:cont56} is satisfied just after the contraction of the vertices $v'_1$ and $v'_2$,
we obtain that at least one of the following holds for the graph $H$:
\begin{itemize}
\item all five skeleton neighbors of $v_1$ are red,
\item the red degree of $w'$ is six,
      $v_1$ is the only skeleton neighbor of $w'$ that is not a red neighbor of $w'$, and
      all skeleton neighbors of $v''_1$ are red, or
\item the red degree of $w'$ in the skeleton is at least five, and
      the graph when the vertex $w'$ was created contained a non-skeleton vertex $u_0$ that
      was adjacent to $w'$, $v_1$ and $v''_1$ and all neighbors of $u_0$ were red.
\end{itemize}
We next exclude the first and the third cases.
In the first case, the red degree of the vertex obtained by contracting $v_1$ and $v_2$
is at least $5+6-4=7$ as all neighbors of the contracted vertex in the skeleton are red.
Hence, the first case does not apply.

We now exclude the third case.
If the vertex $u_0$ has not been contracted to a skeleton vertex in $H$,
then the red degree of $w'$ after contracting the vertices $v_1$ and $v_2$ is at least seven
unless $v_2$ is a common neighbor of $v_1$ and $w'$ in the skeleton of $H$.
If the vertex $v_2$ is a common neighbor of $v_1$ and $w'$,
then the red degree of the vertex obtained by contracting $v_1$ and $v_2$ is at least seven as
the contracted vertex has six red skeleton neighbors (and $u_0$ is also a red neighbor of it).
Hence, the vertex $u_0$ has been contracted to a skeleton vertex of $H$.

If the vertex $u_0$ has been contracted to $w'$, then $w'$ has red degree seven in $H$.
If the vertex $u_0$ has been contracted to a skeleton vertex different from $v_1$ and $v''_1$,
then the skeleton vertex to which $u_0$ has been contracted has red degree seven (note that
$w'$ is the only common neighbor of $v_1$ and $v''_1$ in the skeleton of $H$).
If the vertex $u_0$ has been contracted to $v_1$,
then the contraction of $v_1$ and $v_2$ results in a vertex with red degree at least $5+6-4=7$ (as
all skeleton neighbors of the contracted vertex are red).
The only remaining case is that the vertex $u_0$ has been contracted to $v''_1$.
Unless $v_2$ is a common neighbor of $w'$ and $v_1$,
the red degree of $w'$ after contracting the vertices $v_1$ and $v_2$ is at least seven.
And if the vertex $v_2$ is a common neighbor of $v_1$ and $w'$,
then the red degree of the vertex obtained by contracting $v_1$ and $v_2$ is at least seven as
the contracted vertex has six red neighbors in the skeleton in addition to $v''_1$.

It now remains to analyze the second case.
The vertex $v_2$ must a common neighbor of $w'$ and $v_1$:
otherwise, the degree of $w'$ in the skeleton of $H$ does not change and
the vertex obtained by contracting $v_1$ and $v_2$ is a new red neighbor of $w'$ in the skeleton,
i.e., $w'$ has seven red neighbors in the skeleton of the graph obtained from $H$ by contracting $v_1$ and $v_2$.
Let $z$ be the common neighbor of $v_1$ and $v_2$ different from $w'$, and
let $\alpha$, $\beta$ and $\gamma$ be the non-skeleton vertices forming a triangle in the face of $G_k$
bounded by $v_1$, $v_2$ and $z$ such that
$\alpha$ is adjacent to $v_1$, $\beta$ to $v_2$ and $\gamma$ to $z$ (see Figure~\ref{fig:cont556}).

\begin{figure}
\begin{center}
\epsfbox{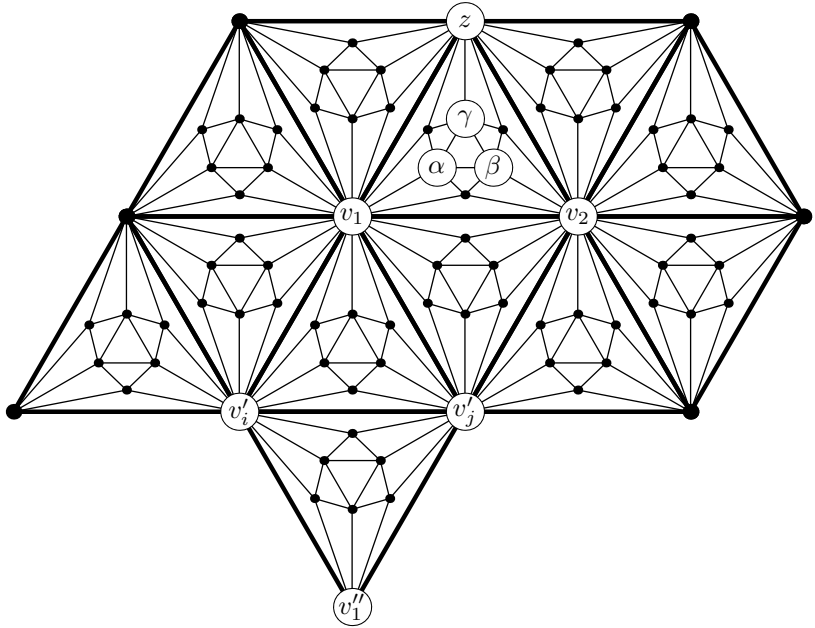}
\end{center}
\caption{Notation used in the proof of Lemma~\ref{lm:cont556} where $\{i,j\}=\{1,2\}$.}
\label{fig:cont556}
\end{figure}

Since $v_2$ is a red neighbor of $w'$,
the vertex obtained by contracting $v_1$ and $v_2$ in $H$ has six red neighbors in the skeleton (all
its neighbors in the skeleton except for $z$);
it follows that
the vertices $\alpha$ and $\beta$ have been contracted to one of the neighbors of $v_1$ and $v_2$ in the skeleton of $H$.
Let $a$ be the vertex to which $\alpha$ has been contracted and $b$ the vertex to which $\beta$ has been contracted.
The vertex $a$ is a neighbor of $v_1$ in the skeleton of $H$ (otherwise, $a$ has red degree at least seven in $H$),
similarly, $b$ is a neighbor of $v_2$ in the skeleton and
neither $a$ nor $b$ is the vertex $v_1$, $v_2$ or $z$ (otherwise,
the vertex obtained by contracting $v_1$ and $v_2$ has red degree at least seven).
In addition, the vertices $a$ and $b$ are the same vertex or adjacent in the skeleton (otherwise,
they have red degree at least seven after contracting $v_1$ and $v_2$).
Since the vertex $b$ cannot be the vertex $w'$ (as $v_1$ is not a red neighbor of $w'$ in $H$),
it follows that $a$ is the vertex $w'$ and $b$ is the common neighbor of $w'$ and $v_2$ different from $v_1$ in the skeleton (as
indicated in Figure~\ref{fig:cont556c}).
Since $\alpha$ has been contracted to $w'$,
the vertex resulting from the contraction of $v_1$ and $v_2$ is a red neighbor of $w'$,
i.e., after the contraction of $v_1$ and $v_2$, all skeleton neighbors of $w'$ are red.

The vertex $\gamma$ has been contracted to a red neighbor of $w'$ in $H$ (otherwise, $w'$ has red degree at least seven) and
this vertex must be a neighbor of $z$ in the skeleton of $H$ or the vertex $v''_1$ (otherwise,
the vertex to which $\gamma$ was contracted has red degree at least seven in $H$).
In the former case,
$\gamma$ has been contracted to $v_2$ since it cannot be contracted to $v_1$ (as
all six skeleton neighbors of $w'$ different from $v_1$ are red neighbors of $w'$).
In the latter case,
the skeleton degree of $v''_1$ is five both at the time of contraction of $\gamma$ and in $H$, and
$v''_1$ is a common neighbor of $w'$ and $b$,
which is the case only if the $5$-vertex $v'_1$ is a common neighbor of $v_1$ and $v_2$.
The contractions of the vertices $\alpha$, $\beta$ and $\gamma$ are indicated in Figure~\ref{fig:cont556c}.

\begin{figure}
\begin{center}
\epsfbox{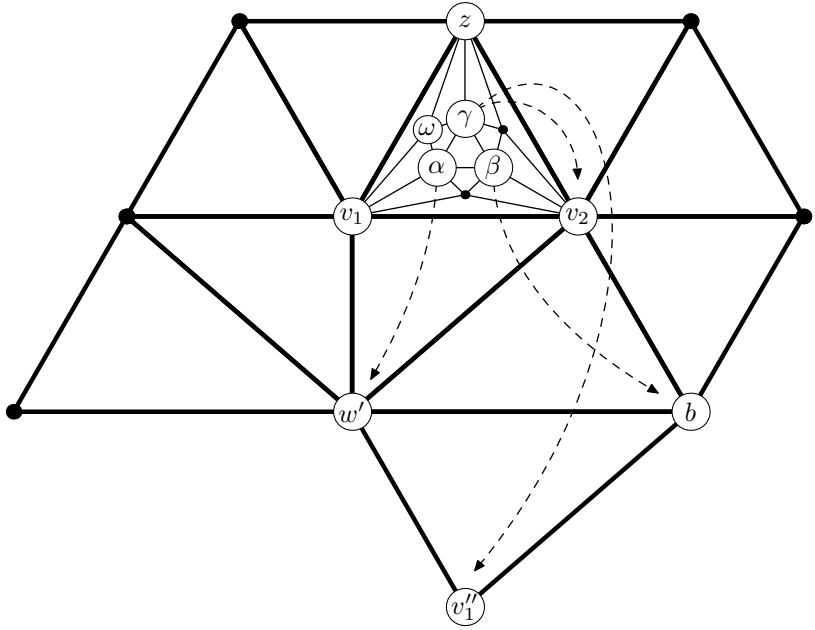}
\end{center}
\caption{Analysis presented in the proof of Lemma~\ref{lm:cont556}.}
\label{fig:cont556c}
\end{figure}

We now analyze the order of 
the contraction of the vertices $v'_1$ and $v'_2$, and
the contractions of the vertices $\alpha$, $\beta$ and $\gamma$ (or the vertices
to which they have been contracted to) to skeleton vertices.
We start by establishing that $\alpha$ was contracted to a skeleton vertex before $\beta$.
If this is not the case,
then the vertex $b$ has red degree at least seven just after $\beta$ is contracted to it:
if $\gamma$ has not yet been contracted, then $b$ has five red skeleton neighbors and
both $\alpha$ and $\gamma$ are also among its red neighbors,
if $\gamma$ has been contracted to $v_2$, then $b$ has six red skeleton neighbors and $\alpha$ is an additional red neighbor, and
if $\gamma$ has been contracted to $v''_1$, then $v''_1$ had five red skeleton neighbors and
$\alpha$ and $\beta$ were additional red neighbors of $v''_1$ just before the contraction $\beta$ to $b$.

We next establish that both $\alpha$ and $\beta$ were contracted to skeleton vertices before the contraction of $v'_1$ and $v'_2$.
If this is not the case, just before the contraction of $\beta$ to a skeleton vertex,
the vertex $w'$ exists and $\alpha$ has been contracted to it (as its contraction precedes that of $\beta$).
However, the red degree of $w'$ must be at least seven at that point as
its red neighbors are six skeleton neighbors and the vertex $\beta$.
We conclude that \emph{first the vertex $\alpha$ is contracted to either $v'_1$ or $v'_2$,
then $\beta$ is contracted to $b$ and then the vertices $v'_1$ and $v'_2$ are contracted}.

Note that either the vertex $v'_1$ or the vertex $v'_2$ is a common neighbor of the vertices $v_1$ and $v_2$,
which is the vertex $v'_j$ in Figure~\ref{fig:cont556}.
We next exclude that the vertex $v'_2$ is a common neighbor of $v_1$ and $v_2$,
i.e., the case when $j=2$ in Figure~\ref{fig:cont556}.
Suppose that $v'_2$ is a common neighbor of $v_1$ and $v_2$.
This implies that $\gamma$ has been contracted to $v_2$ ($\gamma$ can have been contracted to $v''_1$
only if $v'_1$ is a common neighbor of $v_1$ and $v_2$ as we established earlier
when discussing which vertices $\gamma$ can have been contracted to).
If $\alpha$ is contracted to $v'_1$,
then $b$ has at least seven red neighbors just after $\beta$ is contracted to it:
five red skeleton neighbors different from $v_2$,
the vertex $v'_1$ and
either $v_2$ or $\gamma$ (depending on whether $\gamma$ has been contracted to $v_2$ or not).
Hence, $\alpha$ has been contracted to $v'_2$.
Observe that $\gamma$ was contracted to $v_2$ before the contraction of $\alpha$ to $v'_2$ as
otherwise $v'_2$ would have at least seven red neighbors just after $\alpha$ is contracted to it:
five skeleton neighbors different from $v_1$ and the vertices $\beta$ and $\gamma$.
Note that we have established that the order of the contractions was as follows:
$\gamma$ to $v_2$, $\alpha$ to $v'_2$, $\beta$ to $b$, and finally the contraction of $v'_1$ and $v'_2$.

Let $\omega$ be the common neighbor of $v_1$, $z$, $\alpha$ and $\gamma$.
Note that $\omega$ was contracted before the contraction of $\gamma$ to $v_2$ (otherwise, $v_2$ would have seven red neighbors
just after $\gamma$ was contracted to it: five skeleton neighbors, $\alpha$ and $\omega$) and
the vertex $\omega$ was actually contracted to one of the following vertices:
$v_2$, a skeleton neighbor of $v_2$ or $\alpha$ (otherwise,
$v_2$ would have seven red neighbors just after $\gamma$ was contracted to it).
Since $\alpha$ has eventually been contracted to $v'_2$,
the vertex $\omega$ has been contracted to a red neighbor of $w'$ in $H$.
It follows that the vertex $\omega$ has been contracted to either the vertex $v_2$ or to the vertex $b$.
If $\omega$ was contracted to $v_2$,
then just after the contraction of $\omega$ to $v_2$,
the vertex $v_2$ would have four skeleton red neighbors and
$\alpha$, $\beta$ and $\gamma$ would also be its red neighbors,
which is impossible.
If $\omega$ was contracted to $b$,
the vertex $b$ would have six skeleton red neighbors and
$v_1$, $z$, $\alpha$ and $\gamma$ would also be its red neighbors.
We conclude that $v'_1$ is a common neighbor of $v_1$ and $v_2$ (as claimed in the statement of the lemma),
i.e., $j=1$ and $i=2$ in Figure~\ref{fig:cont556}.

Note that the vertex $\alpha$ cannot be contracted to $v'_2$ as
$v'_2$ would have at least seven red neighbors just after $\alpha$ is contracted to it:
five skeleton neighbors different from $v_1$,
the vertex $\beta$, and
either the vertex $v_2$ or $\gamma$ (note that 
if $\gamma$ has been contracted to $v''_1$,
then the contraction of $\gamma$ to $v''_1$ happened only after the contraction of $v'_1$ and $v'_2$).
So, the vertex $\alpha$ was contracted to $v'_1$.
This yields that all neighbors of $v'_1$ in the skeleton possibly except for $v_1$
were red neighbors of $v'_1$ before the contraction of $v'_1$ and $v'_2$.
The proof of the lemma is now finished.
\end{proof}

We are now ready to determine the twin-width of the graph $G_k$ with $k\ge\lbound$.

\begin{theorem}
\label{thm:main}
The twin-width of $G_k$ is $7$ for every $k\ge\lbound$.
\end{theorem}

\begin{proof}
Fix $k\ge\lbound$.
The twin-width of $G_k$ is at most $7$ by Lemma~\ref{lm:upper}.
We next show that the twin-width of $G_k$ is equal to $7$.
Suppose that the twin-width of $G_k$ is at most $6$ and
fix a sequence of $6$-contractions reducing the graph $G_k$ to a single vertex.

We analyze the initial part of this sequence until the first step 
when the contracted vertices $x$ and $y$ satisfy that
both $x$ and $y$ are skeleton vertices and one of the following holds:
\begin{itemize}
\item the vertices $x$ and $y$ are not adjacent,
\item at least one of the vertices $x$ and $y$ resulted from a contraction of two skeleton vertices,
\item neither $x$ nor $y$ is at distance at most two from a $5$-vertex of the skeleton of $G_k$, or
\item there exists a $5$-vertex $z$ of the skeleton of $G_k$ such that
      the contraction of $x$ and $y$
      is the third contraction involving a vertex at distance at most two from $z$ in the skeleton of $G_k$.
\end{itemize}
Observe that the choice of the step implies that
all $6$-contractions till this step are semiplanar and
there is no vertex $v$ such that three or more skeleton vertices of $G_k$ have been contracted to $v$.
In addition, we will show that
the assumptions of Lemmas~\ref{lm:cont-non} and~\ref{lm:cont66} are satisfied till the considered step, and
all $6$-contractions happen ``locally'' around the $5$-vertices of the skeleton of $G_k$.

We will rule out that any of the four cases described above can happen.
Since the assumptions of Lemmas~\ref{lm:cont-non} and~\ref{lm:cont66} are satisfied,
the first or the third cases cannot apply.
Similarly, the second case cannot apply as
it would result in a vertex with red degree
at least $5+6+6-9=8$ (if exactly one of the vertices $x$ and $y$ resulted from a contraction of two skeleton vertices) or
at least $5+6+6+6-14=9$ (if there both $x$ and $y$ resulted from a contraction of two skeleton vertices).
Hence, the last case is the only one that can possibly apply.
We next analyze possible $6$-contractions in the vicinity of each $5$-vertex of the skeleton of $G_k$.

If $z$ is a $5$-vertex of the skeleton of $G_k$,
Lemmas~\ref{lm:cont-non} and~\ref{lm:cont66} imply that
the first contraction involving a skeleton vertex at distance at most two from $z$ and another skeleton vertex
is a contraction of $z$ and a neighbor of $z$ in the skeleton.
This results in the skeleton having a $7$-vertex $w$ adjacent to two $5$-vertices $u$ and $u'$
while the other vertices remain $6$-vertices and
no separating $(\le 4)$-cycle in the skeleton is created.
Hence, after the first contraction involving a skeleton vertex at distance at most two from $z$,
the graph still satisfies the assumptions of Lemmas~\ref{lm:cont-non} and~\ref{lm:cont66}, and
the graph satisfies one of the two conclusions of Lemma~\ref{lm:cont56}.

If the second contraction involving a skeleton vertex at distance at most two from $z$ did not involve $u$ or $u'$,
the second contraction would be the contraction of the two adjacent $6$-vertices that are neighbors of $u$ and $u'$ (one of them
is a common neighbor of $u$ and $w$, and the other is a common neighbor $u'$ and $w$) by Lemma~\ref{lm:cont66} and
the resulting vertex would have red degree at least seven (note that $w$ would be a red neighbor of the resulting vertex).
Hence, the second contraction involving a skeleton vertex at distance at most two from $z$
involves one of the vertices $u$ and $u'$, say $u$, and an adjacent $6$-vertex in the skeleton.
Lemma~\ref{lm:cont556} yields (note that we use $k\ge 7$ here) that
all neighbors of $w$ except for $u$ in the skeleton are red and
all five neighbors of $u'$ in the skeleton are red.
The contraction changes the $7$-vertex $w$ to a $6$-vertex
while creating a new $7$-vertex and a new $5$-vertex $u''$ in the skeleton (see Figure~\ref{fig:main:555}).
Note that $u''$ is not a red neighbor of the new $7$-vertex,
all six skeleton neighbors of $w$ are red, and
the assumptions of Lemmas~\ref{lm:cont-non} and~\ref{lm:cont66} are still satisfied.

\begin{figure}
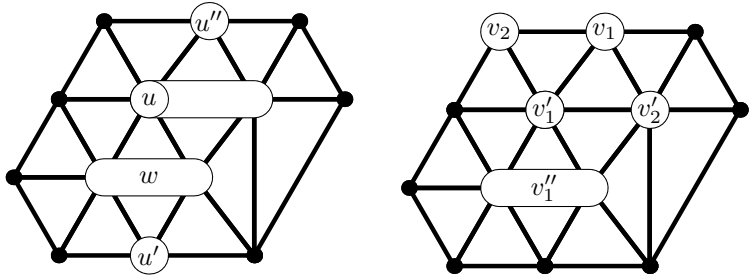

\begin{center}
\epsfbox{twin7-12.mps}
\hskip 5mm
\epsfbox{twin7-13.mps}
\end{center}
\caption{The only possible contractions in the vicinity of a $5$-vertex of the skeleton of $G_k$
         as established in the proof of Theorem~\ref{thm:main}.
         Both configurations represent the same situation:
	 the left with the notation from the proof of Theorem~\ref{thm:main}, and
	 the right with the notation from Lemma~\ref{lm:cont556} as applied with respect to the third contraction.}
\label{fig:main:555}
\end{figure}

The third contraction involving a skeleton vertex at distance at most two from $z$ must involve $u''$ (note that
all five neighbors of $u'$ in the skeleton are red and so its contraction would yield a vertex of red degree $5+6-4=7$).
Lemma~\ref{lm:cont556} can be applied as $k\ge\lbound$ and
the contraction involving $u$ resulted in the first case described in Lemma~\ref{lm:cont56}.
The last point of the conclusion of Lemma~\ref{lm:cont556} yields that
before contracting $u$ and its neighbor,
all neighbors of $u$ in the skeleton possibly except for $u''$ were red.
In particular,
the vertex $w$ was a red neighbor of $u$ before the second contraction,
which is impossible (the remaining six neighbors of $w$ in the skeleton
are red as established earlier).
It follows that the fourth case cannot happen and 
a sequence of $6$-contractions reducing the graph $G_k$ to a single vertex does not exist.
\end{proof}

\section*{Acknowledgement}

The authors would like to thank both anonymous reviewers for carefully reading the paper and
their numerous comments that helped to improve the clarity of the arguments.

\bibliographystyle{bibstyle}
\bibliography{twin7}

\begin{thebibliography}{10}
\providecommand{\url}[1]{\texttt{#1}}
\providecommand{\urlprefix}{URL }
\providecommand{\eprint}[2][]{\url{#2}}

\bibitem{AhnHKO22}
J.~Ahn, K.~Hendrey, D.~Kim and S.~Oum: \emph{Bounds for the twin-width of
  graphs}, SIAM Journal on Discrete Mathematics \textbf{36} (2022), 2352--2366.

\bibitem{BalH21}
J.~Balab\'{a}n and P.~Hlin\v{e}n\'{y}: \emph{Twin-width is linear in the poset
  width}, in: P.~A. Golovach and M.~Zehavi (eds.), 16th International Symposium
  on Parameterized and Exact Computation (IPEC), \emph{LIPIcs}, volume 214
  (2021), article no. 6, 13pp.

\bibitem{BekLHK22}
M.~A. Bekos, G.~D. Lozzo, P.~Hliněný and M.~Kaufmann: \emph{Graph product
  structure for h-framed graphs}, in: S.~W. Bae and H.~Park (eds.), 33rd
  International Symposium on Algorithms and Computation (ISAAC 2022),
  \emph{LIPIcs}, volume 248 (2022), article no. 23, 15pp.

\bibitem{BerBD22}
P.~Berg\'{e}, E.~Bonnet and H.~D\'{e}pr\'{e}s: \emph{Deciding twin-width at
  most 4 is {NP}-complete}, in: M.~Boja\'{n}czyk, E.~Merelli and D.~P. Woodruff
  (eds.), 49th International Colloquium on Automata, Languages, and Programming
  (ICALP), \emph{LIPIcs}, volume 229 (2022), article no. 18, 20pp.

\bibitem{BonCKKLT22}
{\'E}.~Bonnet, D.~Chakraborty, E.~J. Kim, N.~K{\"o}hler, R.~Lopes and
  S.~Thomass{\'e}: \emph{Twin-width {VIII}: Delineation and win-wins}, in:
  H.~Dell and J.~Nederlof (eds.), 17th International Symposium on Parameterized
  and Exact Computation (IPEC 2022), \emph{LIPIcs}, volume 249 (2022), article
  no. 9, 18pp.

\bibitem{BonGKTW21b}
{\'E}.~Bonnet, C.~Geniet, E.~J. Kim, S.~Thomass{\'e} and R.~Watrigant:
  \emph{Twin-width {II}: Small classes}, Proceedings of the 2021 ACM-SIAM
  Symposium on Discrete Algorithms (SODA) (2021), 1977--1996.

\bibitem{BonGKTW21c}
E.~Bonnet, C.~Geniet, E.~J. Kim, S.~Thomass\'{e} and R.~Watrigant:
  \emph{Twin-width {III}: Max independent set, min dominating set, and
  coloring}, in: N.~Bansal, E.~Merelli and J.~Worrell (eds.), 48th
  International Colloquium on Automata, Languages, and Programming (ICALP),
  \emph{LIPIcs}, volume 198 (2021), article no. 35, 20pp.

\bibitem{BonGTT22}
{\'E}.~Bonnet, C.~Geniet, R.~Tessera and S.~Thomass{\'e}: \emph{Twin-width
  {VII}: Groups} (2022), \eprint{arXiv:2204.12330}.

\bibitem{BonGOSTT22}
E.~Bonnet, U.~Giocanti, P.~Ossona~de Mendez, P.~Simon, S.~Thomass\'{e} and
  S.~Toru\'{n}czyk: \emph{Twin-width {IV}: Ordered graphs and matrices},
  Proceedings of the 54th Annual ACM SIGACT Symposium on Theory of Computing
  (STOC) (2022), 924--937.

\bibitem{BonGOT22}
{\'E}.~Bonnet, U.~Giocanti, P.~{Ossona de Mendez} and S.~Thomass{\'e}:
  \emph{Twin-width {V}: Linear minors, modular counting and matrix
  multiplication} (2022), \eprint{arXiv:2209.12023}.

\bibitem{BonKRT22}
{\'E}.~Bonnet, E.~J. Kim, A.~Reinald and S.~Thomass{\'e}: \emph{Twin-width
  {VI}: The lens of contraction sequences}, Proceedings of the 2022 Annual
  ACM-SIAM Symposium on Discrete Algorithms (SODA) (2022), 1036--1056.

\bibitem{BonKRTW21}
{\'E}.~Bonnet, E.~J. Kim, A.~Reinald, S.~Thomass\'e and R.~Watrigant:
  \emph{Twin-width and polynomial kernels}, in: 16th International Symposium on
  Parameterized and Exact Computation (IPEC), \emph{LIPIcs}, volume 214 (2021),
  article no. 10, 16pp.

\bibitem{BonKTW20}
E.~Bonnet, E.~J. Kim, S.~Thomass\'e and R.~Watrigant: \emph{Twin-width {I}:
  Tractable {F}{O} model checking}, Proc. 61th IEEE Annual Symposium on
  Foundations of Computer Science (FOCS) (2020), 601--612.

\bibitem{BonKTW22}
E.~Bonnet, E.~J. Kim, S.~Thomass\'e and R.~Watrigant: \emph{Twin-width {I}:
  Tractable {F}{O} model checking}, Journal of the ACM \textbf{69} (2022),
  article no. 1:3, 46pp.

\bibitem{BonKW22}
{\'E}.~Bonnet, O.~Kwon and D.~R. Wood: \emph{Reduced bandwidth: a qualitative
  strengthening of twin-width in minor-closed classes (and beyond)} (2022),
  \eprint{arXiv:2202.11858}.

\bibitem{BonNOST21}
{\'E}.~Bonnet, J.~Ne\v{s}et\v{r}il, P.~{Ossona de Mendez}, S.~Siebertz and
  S.~Thomass{\'e}: \emph{Twin-width and permutations} (2021),
  \eprint{arXiv:2102.06880}.

\bibitem{Bor89}
O.~V. Borodin: \emph{On the total coloring of planar graphs}, J. Reine Angew.
  Math. \textbf{394} (1989), 180--185.

\bibitem{CraW17}
D.~W. Cranston and D.~B. West: \emph{An introduction to the discharging method
  via graph coloring}, Discrete Mathematics \textbf{340} (2017), 766--793.

\bibitem{DreGJOR22}
J.~Dreier, J.~Gajarsk\'y, Y.~Jiang, P.~{Ossona de Mendez} and J.-F. Raymond:
  \emph{Twin-width and generalized coloring numbers}, Discrete Mathematics
  \textbf{345} (2022), article no. 112746, 8pp.

\bibitem{DvoHJLW21}
Z.~Dvo\v{r}\'ak, T.~Huynh, G.~Joret, C.-H. Liu and D.~R. Wood: \emph{Notes on
  graph product structure theory}, in: D.~R. Wood~({Editor-in-Chief}),
  J.~de~Gier, C.~E. Praeger and T.~Tao (eds.), 2019-20 MATRIX Annals,
  \emph{MATRIX Book Series}, volume~4 (2021), 513--533.

\bibitem{GajPT22}
J.~Gajarsk\'y, M.~Pilipczuk and S.~Toru\'nczyk: \emph{Stable graphs of bounded
  twin-width}, in: C.~Baier (ed.), Proceedings of the 37th Annual ACM/IEEE
  Symposium on Logic in Computer Science (LICS) (2022), article no. 39, 12pp.

\bibitem{Hli22}
P.~Hlin\v{e}n\'{y}: \emph{Twin-width of planar graphs is at most 9, and at most
  6 when bipartite planar} (2022), \eprint{arXiv:2205.05378}.

\bibitem{Hli23}
P.~Hlin\v{e}n\'{y}: \emph{Twin-width of planar graphs; a short proof} (2023),
  \eprint{arXiv:2302.08938}.

\bibitem{HliJ22}
P.~Hlin\v{e}n\'{y} and J.~Jedelsk\'y: \emph{Twin-width of planar graphs is at
  most 8, and at most 6 when bipartite planar} (2022),
  \eprint{arXiv:2210.08620}.

\bibitem{JacP22}
H.~Jacob and M.~Pilipczuk: \emph{Bounding twin-width for bounded-treewidth
  graphs, planar graphs, and bipartite graphs} (2022),
  \eprint{arXiv:2201.09749}.

\bibitem{PilS22}
M.~Pilipczuk and M.~Soko{\l}owski: \emph{Graphs of bounded twin-width are
  quasi-polynomially $\chi$-bounded} (2022), \eprint{arXiv:2202.07608}.

\bibitem{PilSZ22}
M.~Pilipczuk, M.~Soko{\l}owski and A.~Zych-Pawlewicz: \emph{Compact
  representation for matrices of bounded twin-width}, in: P.~Berenbrink and
  B.~Monmege (eds.), 39th International Symposium on Theoretical Aspects of
  Computer Science (STACS), \emph{LIPIcs}, volume 219 (2022), article no. 52,
  14pp.

\bibitem{Tho22}
S.~Thomass\'e: \emph{A brief tour in twin-width}, in: M.~Boja\'{n}czyk,
  E.~Merelli and D.~P. Woodruff (eds.), 49th International Colloquium on
  Automata, Languages, and Programming (ICALP), \emph{LIPIcs}, volume 229
  (2022), article no. 6, 5pp.

\end{thebibliography}

\end{document}